\documentclass{amsart}
\usepackage{euscript}

\newtheorem{Theorem}{Theorem}[section]
\newtheorem{Corollary}[Theorem]{Corollary}
\newtheorem{Lemma}[Theorem]{Lemma}
\newtheorem{Proposition}[Theorem]{Proposition}
\theoremstyle{definition}

\def \R{\mathbb R}
\def \p{\partial}

\def\g{\gamma}
\def \g{\gamma}

\def \<{\langle}
\def \>{\rangle}

\def \E{\mathcal{E}}
\def \tX{\widetilde X}
\def \tx{\tilde x}
\def \e{\epsilon}
\def \o{\omega}
\def \l{\lambda}
\def \S{\mathbb S}
\def \gxy{\gamma_{x,y,\lambda,\omega}}
\def \xy{x,y,\lambda,\omega}

\begin{document}

\title[The inverse problem for the local ray transform]{The inverse problem for the local ray transform}

\author[Hanming Zhou]{Hanming Zhou}
\address{Department of Mathematics, University of Washington, Seattle, WA 98195-4350, USA}
\email{hzhou@math.washington.edu}

\begin{abstract}
In this paper we consider the local X-ray transform for general flows. We extend the results on the local and global invertibility of the geodesic ray transform proved by Uhlmann and Vasy \cite{UV} to the X-ray transform for a general flow. The key improvement is that our argument for the ellipticity of the conjugated operator $A_F$ (which is defined in Section 2) can be applied to flows other than the geodesic flow.
\end{abstract}
\maketitle

\section{Introduction}

Let $X$ be a domain in a Riemannian manifold $(\widetilde{X}, g)$ of dimension $\geq 3$. In this paper we consider the local inverse problem for the X-ray transform for a general flow. A \emph{general flow} $\phi_t=(\gamma,\gamma')$ on $S\widetilde{X}$ (the unit sphere bundle on $\tX$) satisfies the following equation
\begin{equation}\label{flow}
\nabla_{\g'}\g'=E(\g, \g'),
\end{equation}
where $\g$ is parameterized by arclength, $E(z, v)\in T_z\widetilde{X}$ is smooth on $S\tX$. We call the curve $\g$ on $\tX$ that satisfies equation \eqref{flow} an \emph{$\E$-geodesic}. Thus we also call the X-ray transform for a general flow the \emph{$\E$-geodesic transform}, for an $\E$-geodesic $\g$ and a function $f$ on $\tX$, it is defined as
$$If(\g)=\int f(\g(t))dt.$$
We extend $E$ onto $T\tX\backslash o$ (the tangent bundle excluding the zero vector for each fiber) by defining $E(z,v)=|v|_g^2E(z,\frac{v}{|v|_g})$ for $(z,v)\in T\tX\backslash o$, we still denote the extended one by $E$. So $E$ is homogeneous of degree $2$ in $v$. When $E=0$, this gives the usual geodesic flows. Let $z\in\p X$, we say $X$ is \emph{$\mathcal{E}$-convex} at $z$ if 
$$\Lambda(z,v)\geq\<E(z,v), \nu(z)\>$$ 
for all $v\in T_z(\p X)\backslash o$, where $\Lambda$ is the second fundamental form of $\p X$, $\nu(z)$ is the inward unit vector normal to $\p X$ at $z$. If the inequality is strict, then we say $X$ is \emph{strictly $\mathcal{E}$-convex} at $z$. (These terms are not standard, we only use them within this paper.)

For an open set $O\subset\overline{X} (O\cap \p X\neq \emptyset)$, we call $\E$-geodesic segments $\g$ which are contained in $O$ with endpoints at $\p X$ \emph{O-local $\E$-geodesics}; we denote the set of these by $\mathcal M_O$. Thus $\mathcal M_O$ is an open subset of the smooth manifold $\mathcal M$ of all $\E$-geodesics. We then define the \emph{local $\E$-geodesic transform} (or the local X-ray transform for a general flow) of a function $f$ defined on $X$ as the collection $(If)(\g)$ of integrals of $f$ along $\g\in\mathcal M_O$, i.e. as the restriction of the X-ray transform to $\mathcal M_O$. We ask the following question:

\emph{Can we determine $f|_O$, the restriction of $f$ on the open set $O$, by knowing $(If)(\g)$, for all $\g\in\mathcal M_O$?}

In order to state our main theorem in concrete terms, we introduce some notation below. Let $\rho\in C^{\infty}(\tX)$ be a defining function of $\p X$, so that $\rho>0$ in $X$ and $\rho<0$ on $\tX\backslash\overline{X}$, vanishes at $\p X$ with $|\nabla\rho(z)|_g=1$ for $z\in \p X$. Our main theorem is an invertibility result for the local $\E$-geodesic transform on neighborhoods of $p\in \p X$  in $\overline{X}$ of the form $\{\tx>-c\}$, $c>0$, where $\tx$ is a function with $\tx(p)=0$, $d\tx(p)=-d\rho(p)$, see Section 2 for the definition of $\tx$. The statement of our theorem is similar to that of \cite{UV} on the local geodesic ray transform, we now are able to show the same results for the local $\E$-geodesic transform.

\begin{Theorem}
Assume $\overline{X}$ is strictly $\E$-convex at $p\in\p X$, there exists a function $\tx\in C^{\infty}(\tX)$ vanishing at $p$ with $d\tx(p)=-d\rho(p)$ such that for $c>0$ sufficiently small, and with $O_p=\{\tx>-c\}\cap\overline{X}$, the local $\E$-geodesic transform is injective on $H^s(O_p), s\geq 0$.

Further, let $H^s(\mathcal M_{O_p})$ denote the restriction of elements of $H^s(\mathcal M)$ to $\mathcal M_{O_p}$, and for $F>0$ let 
$$H^s_F(O_p)=e^{F/(\tx+c)}H^s(O_p)=\{f\in H^s_{loc}(O_p): e^{-F/(\tx+c)}f\in H^s(O_p)\}.$$
Then for $s\geq 0$ there exists $C>0$ such that for all $f\in H^{s}_F(O_p)$,
$$\|f\|_{H_F^{s-1}(O_p)}\leq C\|If|_{\mathcal M_{O_p}}\|_{H^s(\mathcal M_{O_p})}.$$
\end{Theorem}

It is worthy of remark that for this result we only need to assume convexity near the point $p$. This local result is new even in the case that the metric is conformal to the Euclidean metric. While this large weight $e^{F/(\tx+c)}$ means that the control over $f$ in terms of $If$ is weak at $\tx=-c$. Here $F>0$ can be taken small, but non-vanishing. Further, $\tx$, whose existence is guaranteed by the theorem, is such that $\tx=-c$ is concave from the side of $O_p$.

As an immediate consequence and application of our main theorem, we consider a domain $X$ with compact closure equipped with a boundary defining function $\rho: \overline{X}\rightarrow [0,\infty)$ whose level sets $\rho^{-1}(t), 0\leq t<T$ for some $T>0$, are strictly $\E$-convex (viewed from $\rho^{-1}((t,\infty))$), and $d\rho$ is non-zero on these level sets. We obtain the following corollary on the global injectivity of the $\E$-geodesic transform, which is an analog of the geodesic case in \cite{UV}, the proof is exactly the same.

\begin{Corollary}
For $X$ and $\rho$ as above, if the set $\rho^{-1}([T,\infty))$ has $0$ measure, the global $\E$-geodesic transform is injective on $L^2(X)$, while if $\rho^{-1}([T,\infty))$ has empty interior, the global $\E$-geodesic transform is injective on $H^s(X), s>n/2$.
\end{Corollary}

We also point out that one can derive local reconstruction methods in the form of a Neumann series and global reconstruction methods and stability estimates by a layer stripping algorithm as in \cite{UV}.

The (global) X-ray transform has been extensively studied, among which the simpliest case is the geodesic ray transform. The standard geodesic ray transform, where one integrates a function along straight lines, corresponds to the case of Euclidean metric, which is well-known as the Radon transform and is the basis of medical imaging techniques such as CT and PET; while the case of integration along more general geodesics arises in geophysical imaging and ultrasound imaging. Uniqueness and stability of the geodesic ray transform was shown by Mukhometov \cite{Mu1} on simple surfaces, and also for more general families of curves in two dimension. In \cite{AD}, injectivity result was proved for a general family of curves on Finsler surfaces. The case of geodesics was generalized also for simple manifolds to higher dimensions in \cite{BG, Mu2, MR}. Corresponding problem for magnetic ray transform was studied in \cite{DPSU}. Not much is known for non-simple manifolds, certain results are given in \cite{Da, Sh2, Sh3}. A microlocal analysis of the geodesic ray transform when the exponential map has fold type singularities was done in \cite{SU3}. In dimension $n\geq 3$, the paper \cite{FSU} proves injectivity and stability for the X-ray transform integrating over quite a general class of analytic curves with analytic weights on a class of non-simple manifolds with real-analytic metrics. A completely new approach was provided in \cite{UV} to the uniqueness and stability of the global geodesic ray transform in dimension $n\geq 3$, assuming the manifold is foliated by strictly convex hypersurfaces, an explicit reconstruction procedure was also given.

We remark that Corollary 1.2, as a generalization of the corresponding result in \cite{UV}, provides a new approach to the uniqueness of the global problem for the X-ray transform for a general flow. The only method up to now, except in the real-analytic category \cite{FSU}, \cite{SU2}, has been the use of energy type equalities that introduced by Mukhometov \cite{Mu1} and developed by several authors which are now called ``Pestov identities". The global geometric condition imposed here is a natural analog of the condition $\frac{d}{dr}(r/c(r))>0$ proposed by Herglotz \cite{He} and Wiechert and Zoeppritz \cite{WZ} for an isotropic radial sound speed $c(r)$, and the condition of being foliated by strictly convex hypersurfaces in \cite{UV}. It allows in principle for conjugate points of the flow. A similar condition of foliating by convex hypersurfaces was used in \cite{SU4} to satisfy the pseudoconvexity condition needed for Carleman estimats. 

For the local X-ray transform, Krishnan \cite{Kr} studied the local geodesic ray transform under the assumption that the metric is real-analytic. The only result so far in the smooth category is for the local geodesic ray transform proved by Uhlmann and Vasy \cite{UV}. They defined an operator $A$ which is essentially a `microlocal normal operator' for the geodesic ray transform, such that $Af$ only depends on the integral of $f$ on the elements of $\mathcal M_{O_p}$ (the set of $O_p$-local geodesics). The operator $A$ is an elliptic pseudodifferential operator only in $\tx>-c$. However, it turns out that the exponential conjugate $A_F$ of $A$ is a scattering pseudodifferential operator in Melrose's scattering calculus \cite{Mel} on $\tx\geq -c$, see also \cite[Section 2]{UV}. They showed that $A_F$ is a Fredholm operator and it is invertible for $c$ near $0$, which induces both uniqueness and stability estimates for local geodesic transform. The key ingredient from the geodesic nature of the curves in their paper is that for $z\in\tX$ near point $p$, and geodesics $\gamma$ with $\gamma(0)=z, \frac{d}{dt}(\tx\circ \gamma)|_{t=0}=0$, then $\frac{d^2}{dt^2}(\tx\circ \gamma)|_{t=0}$ induces a positive definite quadratic form $\alpha$ near $p$. This quadratic form $\alpha$ plays a crucial role in showing the invertibility of the principal symbol of $A_F$ at the boundary $\tx=-c$.

Our main theorem is proved along the lines of \cite{UV}. We show that one can prove the invertibility of $A_F$ for general flows without the aid of the existence of some positive definite quadratic form, by using the polar coordinates to calculate the boundary principal symbol of $A_F$ at $\tx=-c$. 
\\

\noindent {\bf Acknowledgements.} The author thanks his advisor Prof. Gunther Uhlmann for suggesting this problem and helpful discussions and for reading the previous version of this paper. The work was partially supported by NSF.
\section{Proof of the main theorem}

Suppose first that $X$ is a domain in $(\widetilde X,g)$, $p\in\p X$, and $\p X$ is strictly $\E$-convex at $p$ (hence near $p$; the convexity assumption will guarantee that locally near $p$ in $\overline{X}$, every two points are joined by a unique $\E$-geodesic lying entirely in $X$ with possible exception of the endpoints ), thus 
$$-\Lambda(p,v)+\<E(p,v), \nu(p)\><0$$
for all $v\in T_p(\p X)\backslash o$. Given $\rho$ a boundary defining function of $\overline{X}$ with $|\nabla\rho(z)|_g=1$ for $z\in\p X$, we have that for vectors $v\in T_p (\p X)\backslash o$,
\begin{center}
$-\Lambda(p,v)+\<E(p,v), \nu(p)\>=$Hess$_p\rho(v,v)+\<E(p,v), \nabla\rho(p)\>$,
\end{center}
here $\nabla\rho|_{\p X}=\nu$ is the inward unit vector normal to $\p X$.
For $f\in C^{\infty}(\tX)$ and $\xi\in T\tX\backslash o$, we define $H_gf(\xi):=\<\nabla f, \xi\>$ and $H^2_gf(\xi):=$Hess$f(\xi,\xi)$, then we have that for vectors $v\in T_p\tX\backslash o$,
$$(H_g\rho)(v)=0 \Rightarrow (H^2_g\rho)(v)+\<E(p, v), \nabla\rho(p)\><0.$$
In particular, by the compactness of the unit sphere and the homogeneity of $H^2_g\rho$ and $E$, there is a neighborhood $U_0$ of $p$ in $\widetilde X$ and $\delta >0, C_0>0$ such that for vectors $v\in T_{U_0}\widetilde X\backslash o$,
$$|(H_g\rho)(v)|<\delta|v|_g \Rightarrow (H^2_g\rho)(v)+\< E(\cdot, v), \nabla\rho\>\leq -C_0|v|^2_g.$$
We want to define a function $\tilde x$ near $p$ such that $\tilde x(p)=0$, the region $\tilde x\geq -c$, $\rho\geq 0$, is compact for $c>0$ small, and the level sets of $\tilde x$ are concave from the side of this region (i.e. the super-level sets of $\tx$). By shrinking $U_0$ if needed, we may assume that it is a coordinate neighborhood of $p$. Concretely we let, for $\epsilon>0$ to be decided, and with $|\cdot |$ the Euclidean norm,
$$\tx(z)=-\rho(z)-\e |z-p|^2;$$
then $\tx\geq -c$ gives $\rho+\e |\cdot-p|^2\leq c$ and thus $\rho\leq c$; further, with $\rho\geq 0$ this gives $|z-p|^2\leq c/\e$. Thus for $c/\e$ sufficiently small, the region $\tx\geq-c, \rho\geq 0$, is compactly contained in $U_0$. Further, for $v\in T_{U_0}\tX\backslash o$, $H_g\tx(v)=-H_g\rho(v)-\e H_g|\cdot-p|^2(v)$, so $H_g\tx(v)=0$ implies $|H_g\rho(v)|<C'\e|v|_g$, so with $\delta>0$ as above there is $\e'>0$ such that for $\e\in(0,\e')$, $H_g\tx(v)=0$ in $U_0$ implies $|H_g\rho(v)|<\delta|v|_g$, and then for $\e<\e'$,
$$H_g^2\tx(v)+\<E(\cdot, v),\nabla\tx\>$$
$$=-(H^2_g\rho(v)+\< E(\cdot, v),\nabla\rho\>)-\e (H_g^2|\cdot-p|^2(v)+\< E(\cdot,v), \nabla |\cdot-p|^2\>)$$
$$\geq (C_0-C^{\prime\prime}\e)|v|^2_g.$$
Thus, there is $\e_0>0$ such that for $\e\in (0,\e_0)$, $H_g^2\tx(v)+\<E(p,v), \nabla\tx(p)\>\geq (C_0/2)|v|^2_g$ at $T_p\tX\backslash o$ when $H_g\tx$ vanishes. Thus taking $c_0>0$ sufficiently small (corresponding to $\e_0$), we have constructed a function $\tx$ defined on a neighborhood $U_0$ of $p$ with concave level sets (from the side of the super-level sets) and such that for $0\leq c\leq c_0$,
$$O_c=\{\tx>-c\}\cap\{\rho\geq 0\}$$
has compact closure in $U_0\cap\tX$ (Here $O_c$ is exactly the $O_p$ in Theorem 1.1).

From now on, we work with $x_c=\tx+c$, which is the boundary defining function of $x_c\geq0$; we supress the $c$ dependence and simply write $x$ in place of $x_c$. For most of the following discussion we completely ignore the actual boundary $\rho=0$; this will only play a role at the end since ellipticity properties only hold in $U_0$ and we need $f$ to be supported in $\rho\geq 0$, ensuring localization, in order to obtain injectivity and stability estimates. 

We complete $x$ to a coordinate system $(x,y)$ on a neighborhood $U_1\subset U_0$ of $p$. We use a fibration by level sets of the function $x$ with non-vanishing differential. Letting $V$ be a vector field orthogonal with respect to $g$ to these level sets with $Vx=1$, and using $\{x=0\}$ as the initial hypersurface, the flow of $V$ (locally) identifies a neighborhood of $\{x=0\}$ with $(-\e,\e)_x\times\{x=0\}$, with the first coordinate being exactly the function $x$ (since time $t$ flow by $V$ changes the value of $x$ to $t$). In particular, choosing coordinates $y_j$ on {x=0}, we obtain coordinates on this neighborhood such that $\p_{y_j}$ and $\p_x$ are orthogonal, i.e. the metric is of the form $f(x,y)dx^2+h(x,y,dy)$ with $f>0$. For each point $(x,y)$ we can parameterize $\E$-geodesics through this point by the unit sphere; the relevant ones for us are `almost tangent' to level sets of $x$, i.e. we are interested in ones with tangent vector $k(\lambda\p_x+\omega\p_y)$, $k(x,y,\l,\o)>0$ (to say have unit length), $\omega\in\mathbb S^{n-2}$ and $\lambda$ relatively small. 

Now, the $\E$-geodesic corresponding to $(z_0, v_0)\in S\tX$, $\g=\g_{z_0, v_0}$, is the projection of the flow $\phi_t$ emanating from $(z_0, v_0)$ (i.e. the intergral curve of $H_g$ through this point). If $f$ is a function on the base space $\tX$ then $\frac{d}{dt}(f\circ\g)(0)=H_gf(\g(0), \g'(0))$, $\frac{d^2}{dt^2}(f\circ\g)(0)=H^2_gf(\g(0), \g'(0))+\<E(\g(0),\g'(0)),\nabla f(\g(0))\>$ (Here $\g$ is parameterized by arclength). By our convexity assumption, at $\frac{d}{dt}(x\circ \g)(0)=0$, $\frac{d^2}{dt^2}(x\circ\g)(0)=H^2_gx(\g(0), \g'(0))+\<E(\g(0),\g'(0)),\nabla x(\g(0))\>>0$.

Now, for a given $\E$-geodesic $\g=\g_{x,y,k\l,k\o}$ with $\g(0)=(x,y),\, \g'(0)=k(\l\p_x+\o\p_y)$ (where $\omega\in\mathbb S^{n-2}$ and $\lambda$ relatively small) that is parameterized by arclength, we change the parameter to make the tangent vector $\g'(0)=\l\p_x+\o\p_y$ and $|\g'|_g\equiv|\g'(0)|_g=\frac{1}{k}$ along the curve. So generally the curve $\g$ is no longer parameterized by arclength, but it still has constant speed (this is crucial, if $If(\g)=0$ under the arclength parameter, then it still equals $0$ under the new parameter). Actually, by investigating the equation \eqref{flow} with the extended bundle map $E$, 
$$\nabla_{c\g'}c\g'=c^2\nabla_{\g'}\g'=c^2E(\g, \g')=E(\g, c\g'), \, c>0$$
i.e. the curve under the new parameter is an $\E$-geodesic of different energy level. We denote the original arclength parameter by $s$, the new parameter by $t$, and $\frac{ds}{dt}>0, \g_s(0)=\g_t(0)$ (i.e. $s=0\iff t=0$), then
$$\frac{d^2}{dt^2}(x\circ \g)=\frac{d^2}{ds^2}(x\circ\g)(\frac{ds}{dt})^2+\frac{d}{ds}(x\circ\g)\frac{d^2s}{dt^2}.$$
In particular, for $\g=\g_{x,y,\l,\o}$ with $\g(0)=(x,y),\,\g'(0)=\l\p_x+\o\p_y$, if $\frac{d}{dt}(x\circ \g)(0)=0$ (i.e. $\l=0$, thus $\frac{d}{ds}(x\circ \g)(0)=0$ too), we have $|\g'(0)|_g=|\o\p_y|_g=1$ (i.e. $s=t$). By convexity assumption, we have 
$$\alpha(x,y,0,\omega,0):=\frac{1}{2}\frac{d^2}{dt^2}(x\circ\g)(0)=\frac{1}{2}\frac{d^2}{ds^2}(x\circ\g)(0)>0$$ (Roughly speaking, the local change of parameter preserves convexity/positivity). \emph{Thus, from now on we use $(x,y,\l,\o)$, instead of $(x,y,k\l,k\o)$, to parameterize the curves $\g$ for sufficiently small $\l$. }

\noindent{\emph{Remark:}} In the case of the geodesic flow, the function $\alpha$ defined above gives a positive definite quadratic form on the tangent space of the level sets of $x$, which plays an important role in \cite{UV}. However, for a general flow, $\alpha$ is no longer a quadratic form, this means we need to modify the arguments of \cite{UV} to make it work for more general flows.

Similar to \cite{UV}, we will work in the following general setting. We consider integrals along a family of curves $\g_{x,y,\l,\o}$ in $\R^n$, $(x,y,\l,\o)\in\R\times\R^{n-1}\times\R\times\S^{n-2}$, depending smoothly on the parameters. Here $\R^{n-1}_y$ could be replaced by an arbitrary manifold and below we make $x$ small, so we are working in a tubular neighborhood of a codimension one submanifold of $\tX$. \emph{However, since the changes in the manifold setting are essentially just notational (we are working on a single coordinate chart), for the sake of clarity we work with $\R^{n}$.} Further, below we work with neighborhood of a compact subset $\{0\}\times K\subset \R_x\times\R_y^{n-1}$; $\g_{x,y,\l,\o}(t)$ would only need to be defined for $(x,y)$ in a fixed neighborhood $U$ of $\{0\}\times K$ and for $|\l|, |t|<\delta_0$, $\delta_0>0$ a fixed constant.

The basic feature we need is that for $x\geq 0$ and for $\l$ sufficiently small, depending on $x$, the curves stay in $[0,\infty)\times \R^{n-1}$. Thus for $x=0$, only $\l=0$ is allowed. We assume that 
$$\gxy(0)=(x,y),\, \g'_{\xy}(0)=(\l,\o),$$
$$\g^{\prime\prime}_{\xy}(t)=2(\alpha(\xy,t), \beta(\xy,t)),$$
with $\alpha, \beta$ smooth and
$$\alpha(0,y,0,\o,0)\geq 2C>0.$$
This implies that if $K\subset\R^{n-1}$ is compact, then for a sufficiently small neighborhood $U$ of $\{0\}\times K$ in $\R^n$ (with compact closure), and $|\l|,|t|<\delta_0$, where $\delta_0>0$ small, we have 
$$\alpha(\xy,t)\geq C>0.$$
One may assume that $x<\delta_0$ on $U$. Thus for $\g(t)=(x'(t),y'(t))$,
$$x'=x+\l t+t^2\int^{1}_{0}(1-\sigma)\alpha(\xy,\sigma)\, d\sigma\geq x+\l t+Ct^2/2,$$
so if $|\l|, |t|<\delta_0, (x,y)\in U$, then 
$$x'\geq \frac{C}{2}(t+\frac{\l}{C})^2+(x-\frac{\l^2}{2C}).$$
Thus, for $|\l|\leq \sqrt{2Cx}$ (and $|\l|<\delta_0$), $x'\geq 0$, i.e. the curves ($\E$-geodesics) remain in the half-space $x'\geq 0$ at least for $|t|<\delta_0$. Further, if we fix $x_0>0$, then $x'\geq x_0$ provided $|t+\frac{\l}{C}|>\sqrt{2x_0/C}$ and $|t|<\delta_0$, thus when $|\l|\leq C_0x_0$ and $|\l|<\delta_0$, then $x'\geq x_0$ provided $|t|>\frac{C_0x_0}{C}+\sqrt{2x_0/C}, |t|<\delta_0$. Assuming $x\leq x_0$ and taking $x_0$ sufficiently small so that $\frac{C_0}{C}x_0+\sqrt{2x_0/C}<\delta_0$, we thus deduce that the curve segments $\gxy|_{(-\delta_0,\delta_0)}$ are outside the region $x'<x_0$ for $t$ outside a fixed compact subinterval of $(-\delta_0,\delta_0)$. \emph{From now on, by $\g$ we mean the restriction $\gxy|_{(-\delta_0,\delta_0)}$, and we assume that the functions we integrate along $\g$ are supported in $x'\leq x_0/2$, so all integrals are on a fixed compact subinterval.}

Similar to the facts in \cite{FSU} and \cite{UV}, the maps
$$\Gamma_+: S\tX\times [0,\infty)\rightarrow [\tX\times\tX; diag],\, \Gamma_+(z,v,t)=\g_{z,v}(t)=(z, |z'-z|, \frac{z'-z}{|z'-z|})$$
and
$$\Gamma_-: S\tX\times (-\infty,0]\rightarrow [\tX\times\tX; diag],\, \Gamma_-(z,v,t)=\g_{z,v}(t)=(z, -|z'-z|, -\frac{z'-z}{|z'-z|})$$
are two diffeomorphisms near $S\tX\times \{0\}$. We actually work with (locally, in the region of interest) the set of tangent vectors of the form $\l\p_x+\o\p_y$, where $\o\in\S^{n-2}$ and $\l$ relatively small, which can be regarded as a small perturbation of a portion of the sphere bundle $S\tX$. Thus we reduce $\delta_0$ if necessary so that $\Gamma_+$ is a diffeomorphism on $U_{x,y}\times (-\delta_0,\delta_0)_{\l}\times\S^{n-2}_{\o}\times [0,\delta_0)_t$, and analogously for $\Gamma_-$. We assume this from now on.

Our inverse problem is now that assuming $(If)(\xy)=\int_{\R} f(\gxy(t)) dt$ is known, we would like to recover $f$ from it. Recall our convention from above, the integral is really over $(-\delta_0,\delta_0)$, and $f(x',y')$ vanishes for $x'\geq x_0/2$. Different from the geodesic flow in \cite{UV}, generally the flow defined by $E$ is not time reversible, i.e. 
$$\g_{x,y,-\l,-\o}(-t)\neq \gxy(t).$$
Thus we will have two curves with a given tangent line at each $(x,y)$, so having the integral of functions along both.

The idea is similar to \cite{UV}, namely to consider for $x>0$
\begin{equation}
(Af)(x,y)=\int_{\R}\int_{\S^{n-2}} (If)(\xy)\tilde{\chi}(x,\l) d\l d\o,
\end{equation}
where $\tilde\chi$ is supported in $|\l|\leq \sqrt{2Cx}$. We define 
$$\tilde\chi(x,\l)=x^{-1}\chi(\l/x),$$
where $\chi$ is compactly supported (for sufficiently small $x$). We can allow $\chi$ to depend smoothly on $y$ and $\o$; over compact sets such a behavior is uniform. As what mentioned in \cite{UV}, for $s\geq 0$, 
$$A=L\circ I: H^s([0,\infty)\times \R^{n-1})\rightarrow x^{-s-1}H^s([0,\infty)\times\R^{n-1})$$
is bounded, where $L: H^s([0,\infty)\times\R^{n-1}\times\R\times\S^{n-2})\rightarrow x^{-s-1}H^s([0,\infty)\times\R^{n-1})$ is also bounded and 
$$(Lu)(x,y)=\int_{\R}\int_{\S^{n-2}} u(x,y,\l,\o)\tilde{\chi}(x,\l) d\l d\o.$$ If we show $A$ is invertible as a map between above spaces of functions supported near $x=0$, we obtain an estimate for $f$ in terms of $If$. 

Similar to \cite[Proposition 3.3]{UV}, we want to study the uniform behavior of $A$ as $x\rightarrow 0$. It turns out that $A$ is an elliptic pseudodifferential operator in $x>0$ for proper choice of $\chi$; while the conjugates of $A$ by exponential weights are scattering pseudodifferential operators on $x\geq 0$.

\begin{Proposition}
Suppose $\chi\in C^{\infty}_c(\R)$, $\chi\geq 0$ with $\chi>0$ near $0$, then $A\in \Psi^{-1}(x>0)$ is elliptic, and the operator $A_F=x^{-1}e^{-F/x}Ae^{F/x}$ is in $\Psi^{-1,0}_{sc}(x\geq 0)$ for $F>0$.
\end{Proposition}
\begin{proof}
First we work in $x>0$, ignoring the limit $x\rightarrow 0$. The diffeomorphism property of $\Gamma_{\pm}$ allows us to rewrite, with $|dv|$ denoting a smooth measure on the transversal such as $|d\l||d\o|$,
$$Af(z)=\sum_{\bullet=+,-}\int f(z')\tilde\chi(\Gamma_{\bullet}^{-1}(z,z'))(\Gamma_{\bullet}^{-1})^*(|dv|dt)$$
in terms of $z,z'$ as
\begin{equation}\label{A}
\int f(z')|z'-z|^{-n+1}\{b(z,|z'-z|,\frac{z'-z}{|z'-z|})+b(z,-|z'-z|,-\frac{z'-z}{|z'-z|})\}dz', 
\end{equation}
$$b(z,0,w)=\tilde\chi(z,w)\sigma(z,w),$$
where $\sigma>0$ is bounded below (it is actually the Jacobian factor). Thus $A$ is a pseudodifferential operator with principal symbol given by the Fourier transform of 
$$|z'-z|^{-n+1}\{(\tilde{\chi}\sigma)(z,\frac{z'-z}{|z'-z|})+(\tilde{\chi}\sigma)(z,-\frac{z'-z}{|z'-z|})\}$$
in $Z=z'-z$. By similar argument as in the proof of \cite[Proposition 3.3]{UV}, we can get that the principal symbol of $A$ at $(z,\zeta)$ is of the form
\begin{equation}\label{symbolA}
c_n|\zeta|^{-1}\int_{\S^{n-2}} (\tilde\chi\sigma)(z,\hat Z^{\perp})+(\tilde\chi\sigma)(z,-\hat Z^{\perp}) d\hat Z^{\perp}=2c_n|\zeta|^{-1}\int_{\S^{n-2}} (\tilde\chi\sigma)(z,\hat Z^{\perp}) d\hat Z^{\perp},
\end{equation}
where $\hat Z^{\perp}$ is the parameter for the sphere $\S^{n-2}$ which is orthogonal to $\zeta$.
In particular, under our assumption of $\chi$, $A$ is an elliptic pseudodifferential operator of order $-1$, in accordance with the results of Stefanov and Uhlmann \cite{SU1}, provided $n>2$.

We now turn to the scattering behavior, i.e. as at least one of $x, x' \rightarrow 0$. We apply the scattering coordinates $(x,y,X,Y)$ introduced in \cite[Section 2]{UV}, where 
$$X=\frac{x'-x}{x^2}, \, Y=\frac{y'-y}{x}.$$
With $K$ denoting the Schwartz kernel of $A$, $A_F$ has Schwartz kernel 
\begin{equation}
K_F(x,y,X,Y)=x^{-1}e^{-FX/(1+xX)}K(x,y,X,Y),
\end{equation}
here $K$ has polynomial bounds in terms of $X,Y$. Our main claim is that $K_F$ and its derivatives have exponential decay for $F>0$, and $K_F$ is smooth for $(X,Y)$ finite, non-zero, conormal to $(X,Y)=0$. This will imply that $A_F$ is a scattering pseudodifferential operator.

Similar to \cite{UV}, we can also use $(x,y,|y'-y|,\frac{x'-x}{|y'-y|},\frac{y'-y}{|y'-y|})$ as the local coordinates on $\Gamma_+(supp\tilde\chi\times[0,\delta_0))$; and analogously for $\Gamma_-(supp\tilde\chi\times (-\delta_0,0])$ the coordinates are $(x,y,-|y'-y|,-\frac{x'-x}{|y'-y|},-\frac{y'-y}{|y'-y|})$. Indeed, this corresponds to the fact that we are using $(x,y,\l,\o)$ with $\o\in \S^{n-2}$, instead of $S\tX$, to parameterize curves, when $|y'-y|$ is large relative to $x'-x$, i.e. in our region of interest. Now, in terms of the scattering coordinates,
$$|y'-y|=x|Y|, \, \frac{x'-x}{|y'-y|}=\frac{xX}{|Y|}, \, \frac{y'-y}{|y'-y|}=\hat Y,$$
Thus similar to the estimates in \cite{UV} by Taylor expansion,
\begin{equation}\label{lambda}
\begin{split}
\frac{\l(\Gamma_+^{-1})}{x}=\frac{X}{|Y|}+|Y|\tilde\Lambda(x,y,x|Y|,\frac{xX}{|Y|},\hat Y), \\
\frac{\l(\Gamma_-^{-1})}{x}=-\frac{X}{|Y|}-|Y|\tilde\Lambda(x,y,-x|Y|,-\frac{xX}{|Y|},-\hat Y).
\end{split}
\end{equation}
Similarly,
$$\o(\Gamma_+^{-1})=\hat Y+x|Y|\tilde\Omega(x,y,x|Y|,\frac{xX}{|Y|},\hat Y),$$
$$\o(\Gamma_-^{-1})=-\hat Y-x|Y|\tilde\Omega(x,y,-x|Y|,-\frac{xX}{|Y|},-\hat Y)$$
and
\begin{equation}\label{time}
\begin{split}
t(\Gamma_+^{-1})=x|Y|+x^2|Y|^2\tilde T(x,y,x|Y|,\frac{xX}{|Y|},\hat Y), \\
t(\Gamma_-^{-1})=-x|Y|-x^2|Y|^2\tilde T(x,y,-x|Y|,-\frac{xX}{|Y|},-\hat Y).
\end{split}
\end{equation}
Thus,
\begin{equation}
\begin{split}
dt\, d\l\, d\o(\Gamma_{\pm}^{-1})=J(x,y,\pm |Y|,\pm\frac{X}{|Y|},\pm\hat Y)x^2|Y|^{-1}dXd|Y|d\hat Y\\
=J(x,y,\pm |Y|,\pm\frac{X}{|Y|},\pm\hat Y)x^2|Y|^{-n+1}dXdY
\end{split}
\end{equation}
where the density factor $J$ is smooth and positive, and $J|_{x=0}=1$. Note that on the blow-up of the scattering diagonal, $\{X=0,Y=0\}$, in the region $|Y|>\epsilon |X|$, thus on the support of $\chi$ in view of \eqref{lambda},
\begin{center}
$(x,y,|Y|,\frac{X}{|Y|},\hat Y)$ and $(x,y,-|Y|,-\frac{X}{|Y|},-\hat Y)$
\end{center}
are valid coordinates, with $\pm |Y|$ being the defining functions of the front face of this blow up. Taking into account the $x^{-1}$ in the definition of $\tilde\chi$, we deduce that $K_F$ is given by 
\begin{equation}\label{KF}
\begin{split}
e^{-FX/(1+xX)}|Y|^{-n+1}\{\chi(\frac{X}{|Y|}+|Y|\tilde\Lambda(x,y,x|Y|,\frac{xX}{|Y|},\hat Y)) J(x,y,|Y|,\frac{X}{|Y|},\hat Y)\\
+\chi(-\frac{X}{|Y|}-|Y|\tilde\Lambda(x,y,-x|Y|,-\frac{xX}{|Y|},-\hat Y))J(x,y,-|Y|,-\frac{X}{|Y|},-\hat Y)\},
\end{split}
\end{equation}
so in particular it is conormal to the front face on the blow-up of the scattering diagonal, of the form $\rho^{-n+1}b$, where $b$ is smooth up to the front face, and without the first exponential factor it, together with its derivatives has polynomial growth as $(X,Y)\rightarrow \infty$. We decompose $K_F$ into two pieces supported in $|(X,Y)|<2$ and $|(X,Y)|>1$ by a partition of unity. For the first term, supported in $|(X,Y)|<2$, similar calculations as in \eqref{A} in Fourier transforming this in $(X,Y)$ show that this term of $K_F$ is indeed the Schwartz kernel of an element of $\Psi_{sc}^{-1,0}$, with standard principal symbol being given by the analogue of \eqref{symbolA}. So we will show that the second term is Schwartz in $(X,Y)$ due to the exponential decay of the first factor of \eqref{KF} on the support of $\chi$. 

We use \eqref{time} to express $\frac{\l}{x}$ by 
\begin{equation}
x'=x+\l t+\alpha(x,y,\l,\o)t^2+O(t^3), \, y'=y+\o t+O(t^2),
\end{equation}
where the terms $O(t^2)$ and $O(t^3)$ have coefficients which are smooth in $(x,y,\l,\o)$.
Thus,
$$X=\frac{x'-x}{x^2}=\frac{\l t}{x^2}+\frac{\alpha t^2}{x^2}+\frac{t^3}{x^2}\Upsilon(x,y,\l,\o,t),$$
with $\Upsilon$ a smooth function of its arguments, so by \eqref{time}
$$X=\frac{\l(\Gamma_+^{-1})}{x}|Y|(1+x|Y|\tilde T)+\alpha(\Gamma_+^{-1})|Y|^2(1+x|Y|\tilde T)^2+x|Y|^3\Upsilon(\Gamma_+^{-1}),$$
also
$$X=\frac{\l(\Gamma_-^{-1})}{x}|Y|(-1-x|Y|\tilde T)+\alpha(\Gamma_-^{-1})|Y|^2(-1-x|Y|\tilde T)^2-x|Y|^3\Upsilon(\Gamma_-^{-1}).$$
Thus
\begin{equation}\label{lambda/x}
\begin{split}
\frac{\l(\Gamma_+^{-1})}{x}=\frac{X-\alpha(\Gamma_+^{-1})|Y|^2}{|Y|}+O(x),\\
\frac{\l(\Gamma_-^{-1})}{x}=\frac{-X+\alpha(\Gamma_-^{-1})|Y|^2}{|Y|}+O'(x),
\end{split}
\end{equation}
where $O(x)$ and $O'(x)$ have smooth coefficients in terms of $x,y,x|Y|,xX/|Y|,\hat{Y}$.
Thus for $\frac{\l_{\pm}}{x}\in(-c,c)$, $-c|Y|<X-\alpha(\Gamma_{\pm}^{-1})|Y|^2<c|Y|$, which shows (by the positivity of $\alpha$) that $X\rightarrow +\infty$ on supp$\tilde\chi$ if $|Y|\rightarrow \infty$, and indeed for $|Y|$ large enough, $X>C|Y|^2$ for some $C>0$.

Now for all $N$ the exponential factor in \eqref{KF} is $\leq C|(X,Y)|^{-N}$ for suitable $C$ on the support of $\chi$, so combined with the polynomial estimates for the derivatives of the other factors, it follows that $K_F$ is smooth in $(x,y)$ with values on functions Schwartz in $(X,Y)$ for $(X,Y)\neq 0$, and conormal to $(X,Y)=0$, which is exactly the characterization of the Schwartz kernel of a scattering pseudodifferentail operator. This finishes the proof of the Proposition.

\end{proof}

Here the additional information is the behavior of $A_F$ at $x=0$, but given that $A_F$ is an element of $\Psi_{sc}^{-1,0}$, the same information can be obtained from computing the boundary principal symbol. Indeed this is immediate from \eqref{KF} and \eqref{lambda/x} which show that at $x=0$ (the scattering front face) the Schwartz kernel $\tilde K(y,X,Y)$ of $A_F$ is 
$$e^{-FX}|Y|^{-n+1}\{\chi(\frac{X-\alpha(0,y,0,\hat Y)|Y|^2}{|Y|})+\chi(\frac{-X+\alpha(0,y,0,-\hat Y)|Y|^2}{|Y|})\}.$$
So the desired invertibility of $A_F$ amounts to the Fourier transform of the kernel, $\tilde K(y,\cdot,\cdot)$ being bounded below in absolute value by $c\<(\xi,\eta)\>^{-1}, c>0$ (here $(\xi,\eta)$ are the Fourier dual variables of $(X,Y)$). Thus our job now is to compute the Fourier transform of $\tilde K(y,\cdot,\cdot)$. Note that we allow that $\chi$ also depends on $y$ and $\o$. We have thus shown

\begin{Lemma}
The boundary principal symbol of $x^{-1}e^{-F/x}Ae^{F/x}$ is the $(X,Y)$-Fourier transform of 
\begin{equation}\label{boundary}
\begin{split}
\tilde{K}(y,X,Y)=e^{-FX}|Y|^{-n+1}\{\chi(\frac{X-\alpha(0,y,0,\hat Y)|Y|^2}{|Y|},y,\hat{Y})\\
+\chi(\frac{-X+\alpha(0,y,0,-\hat Y)|Y|^2}{|Y|},y,-\hat{Y})\}.
\end{split}
\end{equation}
\end{Lemma}

In order to find a suitable $\chi$ to make $A_F$ invertible, we follow the same strategy of \cite{UV}, namely we do calculation for $\chi(s)=e^{-s^2/(2F^{-1}\alpha)}$ with $F>0$ ( here we need the positivity of $\alpha$), so $\hat\chi(\cdot)=c\sqrt{F^{-1}\alpha}e^{-F^{-1}\alpha|\cdot|^2/2}$ for appropriate $c>0$. Here $\chi$ does not have compact support, and an approximation argument will be necessary. First, the Fourier transform of $\tilde K$ in $X$ is 
\begin{equation}
\begin{split}\label{FX1}
\mathcal{F}_X\tilde{K}(y,\xi,Y)=|Y|^{2-n}\{e^{-\alpha_+F|Y|^2}e^{-i\alpha_+\xi|Y|^2}\hat\chi((\xi-iF)|Y|,y,\hat Y) \\
+e^{-\alpha_-F|Y|^2}e^{-i\alpha_-\xi|Y|^2}\hat\chi((\xi-iF)|Y|,y,-\hat Y)\},
\end{split}
\end{equation}
where $\alpha_+=\alpha(0,y,0,\hat Y), \, \alpha_-=\alpha(0,y,0,-\hat Y)$. Substituting the particular $\chi$ into \eqref{FX1} yields a non-zero multiple of 
\begin{equation}\label{FX2}
|Y|^{2-n}\{(F^{-1}\alpha_+)^{\frac{1}{2}}e^{-F^{-1}(\xi^2+F^2)\alpha_+|Y|^2/2}+(F^{-1}\alpha_-)^{\frac{1}{2}}e^{-F^{-1}(\xi^2+F^2)\alpha_-|Y|^2/2}\}.
\end{equation}

Now we estimate the Fourier transform of \eqref{FX2} in $Y$ upto some non-zero multiple. As remarked previously, not like the case of the geodesic flow in \cite{UV}, generally our $\alpha$ may contain terms other than a quadratic form in $\hat Y$, which means the exponential term of $\mathcal{F}_X\tilde{K}(y,\xi,Y)$ is not Gaussian like in $Y$, thus we use polar coordiantes to compute the Fourier transform in $Y$. We denote $\frac{F^{-1}(\xi^2+F^2)}{2}$ by $b$, then
$$\mathcal{F}_Y(\mathcal{F}_X\tilde K)(y,\xi,\eta)\simeq\int e^{-i\eta\cdot Y}|Y|^{2-n}(\alpha_+^{1/2}e^{-b\alpha_+|Y|^2}+\alpha_-^{1/2}e^{-b\alpha_-|Y|^2}) dY$$
$$=\int_0^{+\infty}\int_{\S^{n-2}} e^{-i\eta\cdot\hat{Y}|Y|}|Y|^{2-n}(\alpha_+^{1/2}e^{-b\alpha_+|Y|^2}+\alpha_-^{1/2}e^{-b\alpha_-|Y|^2})|Y|^{n-2}d|Y|d\hat Y$$
$$=\frac{1}{2}\int_{\R}\int_{\S^{n-2}} e^{-i\eta\cdot\hat{Y}t}(\alpha_+^{1/2}e^{-b\alpha_+t^2}+\alpha_-^{1/2}e^{-b\alpha_-t^2})dt d\hat Y$$
$$\simeq\frac{1}{2}\int_{\S^{n-2}} b^{-1/2}(e^{-|\eta\cdot\hat Y|^2/4b\alpha_+}+e^{-|\eta\cdot\hat Y|^2/4b\alpha_-})d\hat Y$$
$$=b^{-1/2}\int_{\S^{n-2}} e^{-|\eta\cdot\hat Y|^2/4b\alpha(y,\hat Y)}d\hat Y,$$
here $\simeq$ means equal up to some multiple.

To estimate the Fourier transform of \eqref{boundary}, we need to study the joint $(\xi,\eta)$-behavior, i.e. when $\<(\xi, \eta)\>$ is going to infinity, where we need lower bounds. we denote $(\xi^2+F^2)^{1/2}$ by $\<\xi\>$, then $\mathcal F_{X,Y}\tilde K(y,\xi,\eta)$ is a constant multiple of
$$\<\xi\>^{-1}\int_ {\S^{n-2}} e^{-|\frac{\eta}{\<\xi\>}\cdot\hat Y|^2/4c\alpha(y,\hat Y)}d\hat Y$$
with $c=c(F)>0$. The only property of $\alpha(y,\hat Y)$ we need is locally uniformly positivity, i.e. $0< c_1\leq \alpha\leq c_2$ for some positive constants $c_1, c_2$ that depend on $y$ and are locally uniform, which is true under our assumption. 

When $\frac{|\eta|}{\<\xi\>}\leq C$, then $\<\xi\>^{-1}$ is equivalent to $\<(\xi,\eta)\>^{-1}$ in this region in terms of decay rates, 
$$\int_ {\S^{n-2}} e^{-|\frac{\eta}{\<\xi\>}\cdot\hat Y|^2/4c\alpha(y,\hat Y)}d\hat Y\geq \int_ {\S^{n-2}} e^{-c'|\frac{\eta}{\<\xi\>}|^2} d\hat Y\geq \int_{\S^{n-2}}e^{-C} d\hat Y=C.$$
Thus $\mathcal F_{X,Y}\tilde K(y,\xi,\eta)\geq C\<\xi\>^{-1}\simeq C\<(\xi,\eta)\>^{-1}$. 

When $\frac{|\eta|}{\<\xi\>}$ is bounded from below, in which case $\<(\xi,\eta)\>^{-1}$ is equivalent to $|\eta|^{-1}$, we write $\hat Y=(\hat Y^{\parallel}, \hat Y^{\perp})$ according to the orthogonal decomposition of $\hat{Y}$ relative to $\frac{\eta}{|\eta|}$, where $\hat Y^{\parallel}=\hat Y\cdot \frac{\eta}{|\eta|}$, and $d\hat Y$ is of the form $a(\hat Y^{\parallel})d\hat Y^{\parallel}d\theta$ with $\theta=\frac{\hat Y^{\perp}}{|\hat Y^{\perp}|}, a(0)=1$ then
$$\int_ {\S^{n-2}} e^{-|\frac{\eta}{\<\xi\>}\cdot\hat Y|^2/4c\alpha(y,\hat Y)}d\hat Y\geq \int_ {\S^{n-2}} e^{-c'|\frac{\eta}{\<\xi\>}\cdot\hat Y|^2}d\hat Y$$
$$=\int_{\R}\int_{\S^{n-3}} e^{-c'(\hat Y^{\parallel}\frac{|\eta|}{\<\xi\>})^2} a(\hat Y^{\parallel})d\theta d\hat Y^{\parallel}$$
\begin{equation}\label{Dirac}
=\frac{\<\xi\>}{|\eta|}\int_{\R}\{\frac{|\eta|}{\<\xi\>}e^{-c'(\hat Y^{\parallel}\frac{|\eta|}{\<\xi\>})^2}\}a(\hat Y^{\parallel})d\hat Y^{\parallel}\int_{\S^{n-3}} d\theta.
\end{equation}
Since $\frac{|\eta|}{\<\xi\>}e^{-c'(\hat Y^{\parallel}\frac{|\eta|}{\<\xi\>})^2}\rightarrow \delta_0$ in distributions as $\frac{|\eta|}{\<\xi\>}\rightarrow\infty$, \eqref{Dirac} is equal to $\frac{\<\xi\>}{|\eta|}\int_{\S^{n-3}} d\theta=2C\frac{\<\xi\>}{|\eta|}$ ($C>0$) modulo terms decaying faster as $\frac{|\eta|}{\<\xi\>}\rightarrow\infty$. In particular, there is $N>0$, such that 
$$\int_{\R}\{\frac{|\eta|}{\<\xi\>}e^{-c'(\hat Y^{\parallel}\frac{|\eta|}{\<\xi\>})^2}\}a(\hat Y^{\parallel})d\hat Y^{\parallel}\int_{\S^{n-3}} d\theta\geq C$$
for $\frac{|\eta|}{\<\xi\>}\geq N$. (Notice that the integral on $\S^{n-3}$ uses very strongly the assumption $n\geq 3$; when $n=3$, $d\theta$ is the point measure) Thus $\mathcal F_{X,Y}\tilde K(y,\xi,\eta)\geq C\frac{1}{\<\xi\>}\frac{\<\xi\>}{|\eta|}=C|\eta|^{-1}\simeq C\<(\xi,\eta)\>^{-1}$.

Thus we deduce that $\mathcal F_{X,Y}\tilde K(y,\xi,\eta)\geq c\<(\xi,\eta)\>^{-1}$ for some $c>0$, i.e the ellipticity claim for the case that $\chi$ is a Gaussian. By an approximation argument as in \cite{UV}, one obtain some $\chi\in C^{\infty}_c(\R)$ such that the Fourier transform of $\tilde K$ with this $\chi$ still has lower bounds $\tilde c\<(\xi,\eta)\>^{-1}, \tilde c>0$, as desired. We have thus proved:
\begin{Proposition}
For $F>0$ there exists $\chi\in C^{\infty}_c(\R)$, $\chi\geq 0, \chi(0)=1$, such that for the corresponding operator $x^{-1}e^{-F/x}Ae^{F/x}$ the boundary principal symbol is elliptic.
\end{Proposition}
Hence, we have 
$$A_F=x^{-1}e^{-F/x}Ae^{F/x}\in\Psi_{sc}^{-1,0}$$
is elliptic both in the sense of the standard principal symbol (in the set of interest $O_c$), and the scattering principal symbol, which is at $x=0$. In particular the results of \cite[Section 2]{UV} on elliptic scattering pseudodifferential operators are applicable. Thus, for $c$ small, if we denote $M_c=\{x_c>0\}$, $H^{s,r}_{sc}(M_c)_K=\{f\in H^{s,r}_{sc}(M_c): supp f\subseteq K$ $(K$ is a compact subset of $O_c)\}$ (Roughly speaking, $H^{s,r}_{sc}(M_c)=x_c^{-r}H^s(M_c)$, see \cite[Section 2]{UV} for the detailed definition and properties),
$$A=xe^{F/x}A_Fe^{-F/x}: e^{F/x}H^{s,r}_{sc}(M_c)_K\rightarrow e^{F/x}H^{s+1,r+1}_{sc}(M_c)$$
satisfies the estimates
$$\|f\|_{e^{F/x}H^{s,r}_{sc}(M_c)_K}\leq C\|Af\|_{e^{F/x}H^{s+1,r+1}_{sc}(M_c)}.$$

By similar arguement as in the part after \cite[Lemma 3.6]{UV}, we thus deduce for $s\geq 0, \delta>0$
$$\|f\|_{e^{(F+\delta)/x}H^{s-1}(M_c)_K}\leq C\|Af\|_{e^{F/x}H^{s}(M_c)}\leq C'\|If\|_{H^s(\mathcal M_{M_c})}$$
when $f\in H^{s}(M_c)_K$. With $F+\delta$ replaced by $F$ (since both $F>0$ and $\delta>0$ are arbitrary), this completes the proof of the main theorem. Thus we obtain the local injectivity and stability estimates for the X-ray transform for general flows.


\end{document}